\documentclass[a4paper,twoside]{amsart}

\title
[A Positivstellensatz for projective real varieties]
{A Positivstellensatz for projective real varieties}



\author
 {Claus Scheiderer}
\address
 {Fachbereich Mathematik und Statistik \\
 Universit\"at Konstanz \\
 78457 Konstanz \\
 Germany}
\email
 {claus.scheiderer@uni-konstanz.de}

\thanks
  {I am grateful to Amir Ali Ahmadi for inspiring discussions
  around the projective Positivstellensatz.}

\usepackage{amssymb}
\usepackage[mathscr]{eucal}

\usepackage{hyperref}

\theoremstyle{plain}
\newtheorem{prop}{Proposition}[section]
\newtheorem{thm}[prop]{Theorem}
\newtheorem{cor}[prop]{Corollary}
\newtheorem{lem}[prop]{Lemma}

\theoremstyle{definition}
\newtheorem{dfn}[prop]{Definition}
\newtheorem{rem}[prop]{Remark}
\newtheorem{rems}[prop]{Remarks}

\newtheorem{lab}[prop]{}

\newcommand{\isoto}{\overset{\sim}{\to}}

\newcommand{\into}{\hookrightarrow}

\renewcommand{\iff}{\Leftrightarrow}
\renewcommand{\subset}{\subseteq}

\newcommand{\A}{{\mathbb{A}}}

\newcommand{\N}{{\mathbb{N}}}
\renewcommand{\P}{{\mathbb{P}}}
\newcommand{\R}{{\mathbb{R}}}

\newcommand{\m}{{\mathfrak{m}}}
\newcommand{\p}{{\mathfrak{p}}}

\newcommand{\scrO}{{\mathscr O}}

\newcommand{\x}{{\mathtt{x}}}
\newcommand{\y}{{\mathtt{y}}}

\DeclareMathOperator{\Hom}{Hom}
\DeclareMathOperator{\im}{im}
\DeclareMathOperator{\sgn}{sgn}
\DeclareMathOperator{\Spec}{Spec}
\DeclareMathOperator{\Sper}{Sper}
\DeclareMathOperator{\supp}{supp}

\DeclareTextFontCommand{\textnf}{\normalfont}

\newcommand{\reg}{\mathrm{reg}}

\newcommand{\du}{{\scriptscriptstyle\vee}}
\renewcommand{\emptyset}{\varnothing}
\newcommand{\ol}{\overline}

\newcommand{\wt}[1]{\widetilde{#1}}
\renewcommand{\setminus}{\smallsetminus}
\renewcommand{\epsilon}{\varepsilon}
\renewcommand{\theta}{\vartheta}

\newcommand{\ex}{\exists\,}

\newcommand{\bil}[2]{\langle{#1},{#2}\rangle}

\newcommand{\sa}{semi-algebraic}

\begin{document}

\begin{abstract}
Given two positive definite forms $f,\,g\in\R[x_0,\dots,x_n]$, we
prove that $fg^N$ is a sum of squares of forms for all sufficiently
large $N\ge0$. We generalize this result to projective $\R$-varieties
$X$ as follows. Suppose that $X$ is reduced without one-dimensional
irreducible components, and $X(\R)$ is Zariski dense in $X$. Given
everywhere positive global sections $f$ of $L^{\otimes2}$ and $g$ of
$M^{\otimes2}$, where $L$, $M$ are invertible sheaves on $X$ and $M$
is ample, $fg^N$ is a sum of squares of sections of $L\otimes M^
{\otimes N}$ for all large $N\ge0$. In fact we prove a much more
general version with \sa\ constraints, defined by sections of
invertible sheaves. For nonsingular curves and surfaces and
sufficiently regular constraints, the result remains true even if $f$
is just nonnegative. The main tools are local-global principles for
sums of squares, and on the other hand an existence theorem for
totally real global sections of invertible sheaves, which is the
second main result of this paper. For this theorem, $X$ may be
quasi-projective, but again should not have curve components. In
fact, this result is false for curves in general.
\end{abstract}

\maketitle


\section*{Introduction}

Let $f\in\R[x_0,\dots,x_n]$ be a positive definite homogeneous
polynomial. From Stengle's Positivstellensatz it follows that $f$ can
be written as a sum of squares of quotients of forms with positive
definite denominators. In 1995, Reznick \cite{Re} refined this
observation by proving that one can always find such a representation
where the denominators are powers of the form $g=x_0^2+\cdots+x_n^2$.
In other words, there exists an integer $N\ge0$ such that the form
$fg^N$ is a sum of squares of forms. It was observed by the author
(\cite{sch:guide} 2.1.8) that Reznick's theorem can be seen as a
direct consequence of Schm\"udgen's Positivstellensatz \cite{Sm}. As
remarked in \cite{sch:guide}, this argument works in greater
generality and shows, for any two non-constant positive definite
forms $f$ and $g$ with $\deg(g)\mid\deg(f)$, that $fg^{2N}$ is a sum
of squares of forms for sufficiently large $N\ge0$.

The present work resulted from an attempt to remove the degree
restriction in this statement. Before we outline the results of this
paper, we briefly sketch how this can be achieved. By replacing $f$
with $f_1=fl^{2k}$ for suitable $k\ge0$ and some linear form $l$, one
satisfies the degree condition but looses strict positivity. As a
consequence, Schm\"udgen's theorem cannot be applied any more.
Instead one uses the local-global criterion for sums of squares
\cite{sch:surf}, combined with the fact that strictly positive
elements in local rings are sums of squares \cite{sch:local}, to get
a sum of squares representation for $f_1g^{2N}$ and some $N\ge0$. By
an elementary argument, any summand in such a representation is
divisible by $l^{2k}$. So one can cancel this factor and gets a
representation of $fg^{2N}$ as desired.

To find the proper setting for generalizing this proof, one first
observes that the result is genuinely projective, and not affine, in
nature. Leaving technical details aside for the moment, it is quite
clear what a conjectural generalization to projective $\R$-varieties
$X$ and everywhere positive sections of line bundles should be.
Analyzing the proof just outlined for $X=\P^n$, one sees that the
least obvious step in proving such a generalization is to find a
suitable substitute for the linear form $l$. This leads to a
question of its own interest, namely totally real divisors on
$\R$-varieties.

There are two main results in this paper. The first is an existence
theorem for totally real (Cartier) divisors on quasi-projective real
varieties, proved in Sect.~\ref{secexcompsec}. More specifically,
if $X$ is a reduced quasi-projective $\R$-scheme without
one-dimen\-sional irreducible component and with $X(\R)$ Zariski
dense in $X$, and if $D$ is a divisor on $X$, then for any ample
divisor $E$ and any sufficiently large integer $n$ there exists a
totally real (reduced) effective divisor that is linearly equivalent
to $D+nE$. Here we call an effective divisor $D$ totally real if the
$\R$-points of (the closed subscheme) $D$ are Zariski dense in $\supp
(D)$. In fact we prove a more general version of this result which is
relative to any Zariski dense \sa\ subset $S$ of $X(\R)$ (Theorem
\ref{extotreschnitt}). The proof is based on Bertini's first theorem.
As a corollary we obtain that any divisor on $X$ is equivalent to a
difference of two totally real (reduced) effective divisors, and even
to a single such divisor when $X$ is affine; again we provide a more
general version relative to a \sa\ set $S$. For projective curves,
Theorem \ref{extotreschnitt} and its corollaries fail in general, see
the discussion starting in \ref{curvrems1}. The result can be saved
for nonsingular curves under stronger assumptions on~$S$.

The second main result is a general Positivstellensatz for reduced
projective $\R$-schemes $X$ (Theorem \ref{mainthmtent}), proved in
Sect.~\ref{secppss}. Basically it says that a strictly positive
global section $f$ of an invertible sheaf $L^{\otimes2}$ becomes a
sum of squares after multiplication with a sufficiently high even
power of any nowhere (on $X(\R)$) vanishing section of any ample
invertible sheaf. The substitute in the proof for the linear form $l$
above is a totally real section of a suitable invertible sheaf. The
existence of such a section is guaranteed by Theorem
\ref{extotreschnitt}. We have to assume that $X$ has no
one-dimensional irreducible component since \ref{extotreschnitt} is
known to fail for curves. We do not know if the dimension restriction
for the irreducible components of $X$ can be removed in Theorem
\ref{mainthmtent}. Again, the version we prove is considerably more
general, and is a preordering-type statement relative to \sa\
constraints. For nonsingular curves and surfaces we can even prove a
Nichtnegativstellensatz ($f$~is allowed to have zeros), as long as
the constraints are sufficiently regular.

As a very particular concrete application it follows from Theorem
\ref{mainthmtent} that, for any strictly positive form $f\in\R[x_0,
\dots,x_n]$, there exists some odd power $f^{2m+1}$ that is a sum of
squares (Corollary \ref{oddpow}). Our Positivstellensatz (in the
version for $\P^n$) was recently applied by Ahmadi and Parrilo
\cite{AP} to time continuous dynamical systems.

In Sect.~\ref{secdiv} we introduce the concept of totally real
divisors and its generalization relative to a \sa\ set, the weakly
compatible divisors. Moreover we define the sign of a section of an
invertible sheaf $L^{\otimes2}$ at a real point, and we discuss basic
properties of these notions that are used later.

For proving (and even formulating) the results of this paper, it is
necessary to use the language of schemes, although some particular
cases can be phrased in a more naive language. In a few places we
also have to work with the real spectrum. A brief discussion of the
concepts used is given in a first section with preliminaries.


\section{Notations and preliminaries}\label{secprelims}%

\begin{lab}
Let $k$ be a field and $X$ a $k$-scheme of finite type. For any
$k$-algebra $E$ we write $X(E)=\Hom_k(\Spec E,X)$ for the set of
$E$-valued points of $X$. We often confuse an element of $X(k)$ with
its closed image point in $X$. By $X_\reg$ we denote the regular
locus of $X$, i.e.\ the set of $x\in X$ for which the local ring
$\scrO_{X,x}$ is regular; this is an open subset of~$X$.
\end{lab}

\begin{lab}
A general reference for \sa\ sets and the real spectrum is
\cite{BCR}. The real spectrum of the ring $A$ is denoted $\Sper(A)$.
Let $R$ be a real closed field, and let $X$ be a separated $R$-scheme
of finite type. A subset $S$ of $X(R)$ is \sa\ if $S\cap U(R)$ is a
\sa\ subset of $U(R)$ for every open affine subset $U$ of $X$. From
the order topology of $R$ we get a topology on $X(R)$ that is
sometimes referred to as the euclidean topology (as opposed to the
Zariski topology on $X$). If $X$ is affine, then $X(R)$ is in the
natural way a topological subspace of $\Sper R[X]$. With each \sa\
subset $S$ of $X(R)$ one associates a constructible subset $\wt S$ of
$\Sper R[X]$, in such a way that $S=X(R)\cap\wt S$; in fact, $\wt S$
is defined by the same system of inequalities inside $\Sper R[X]$ as
$S$ inside $X(R)$. It is well-known that $\wt S$ is open (resp.\
closed) in $\Sper R[X]$ if (and only if) the same is true for $S$ in
$X(R)$. It would be natural to extend the construction of the real
spectrum and of the tilde operator to non-affine $X$, but this is not
needed here.
\end{lab}

Let $R$ be a real closed field. The following is a well-known
consequence of the Artin-Lang theorem. For lack of a suitable
reference we give a short indication of the proof.

\begin{prop}\label{artinlang}%
Let $X$ be an integral $R$-scheme of finite type, and let $S\subset
X(R)$ be a \sa\ subset. Then $S$ is Zariski dense in $X$ if and only
if $S$ contains a non-empty (euclidean) open subset of $X_\reg(R)$.
\end{prop}

\begin{proof}
One easily reduces to showing: If $X$ is affine and nonsingular, then
for any $0\ne g\in R[X]$, the zero set of $g$ in $X(R)$ has empty
interior. To prove this, assume that $f_1,\dots,f_r\in R[X]$ are such
that $U:=\{\xi\in X(R)\colon f_i(\xi)>0$, $i=1,\dots,r\}$ is nonempty
and $g\equiv0$ on $U$. The function field $R(X)$ of $X$ has an
ordering $\alpha$ which makes $f_1,\dots,f_r$ positive.
By the (refined) Artin-Lang theorem (see for instance \cite{Be} Thm.\
1.3), there exists $\xi\in U$ with $\sgn_\alpha(g)=\sgn g(\xi)$, and
in particular, $g(\xi)\ne0$, contradiction.
\end{proof}

\begin{cor}
For an integral $R$-scheme $X$ of finite type, $X(R)$ is Zariski
dense in $X$ if and only if $X$ has a nonsingular $R$-point, if and
only if the function field $R(X)$ of $X$ is real, i.e.\ can be
ordered.
\end{cor}

\begin{proof}
The equivalence $X_\reg(R)\ne\emptyset$ $\iff$ $R(X)$ real is proved
in \cite{Be}.
\end{proof}


\section{Divisors and invertible sheaves}\label{secdiv}%

In this section let $R$ be a real closed field, and let always $X$ be
a reduced separated $R$-scheme of finite type.

\begin{lab}\label{dfndivwc}%
By a divisor on $X$ we always mean a Cartier divisor. An effective
Cartier divisor on $X$ is the same as a closed subscheme $D$ of $X$
whose sheaf of ideals is everywhere locally generated by one element
that is not a zero divisor. The support $\supp(D)$ of $D$ is the
closed subset of $X$ underlying $D$.

Let $D$ be an effective divisor on $X$, and let $S\subset X(R)$ be a
\sa\ set. We say that $D$ is \emph{weakly compatible with $S$} if
$S\cap D(R)$ is Zariski dense in $\supp(D)$. We call $D$
\emph{totally real} if $D$ is weakly compatible with $S=X(R)$.
\end{lab}

\begin{rem}
In \cite{PS}, the notion of compatibility between a \sa\ set
$S\subset X(R)$ and a prime Weil divisor $Y$ on $X$ was introduced
(in the case where $X$ is normal). The notion from \cite{PS} is
stronger (at least when $S$ is closed) than weak compatibility as
defined in \ref{dfndivwc}, which explains our choice of terminology
here. The existence results for weakly compatible sections proved in
Sect.~\ref{secexcompsec} below remain true for the stronger notion
of compatibility (suitably adapted to the more general situation when
$X$ is not normal). However, that notion is more technical and is not
needed here, which is why we work with the easier concept of weak
compatibility.
\end{rem}

\begin{lab}
Let $L$ be an invertible sheaf (locally free $\scrO_X$-module of rank
one) on $X$. Given a global section $s$ of $L$, let $Z(s)$ be the
closed zero subscheme of $s$.
The section $s$ is said to be \emph{regular} if $Z(s)$ is an
effective divisor on $X$. Since $X$ is reduced, it is equivalent that
$\supp(s)$ does not contain any irreducible component of $X$.
A regular section $s$ of $L$ is called \emph{reduced} if the closed
subscheme $Z(s)$ of $X$ is reduced. We say that the regular section
$s$ is \emph{weakly compatible with $S$} if the effective divisor
$Z(s)$ is weakly compatible with $S$, c.f.\ Definition
\ref{dfndivwc}. Again, we call $s$ \emph{totally real} if $s$ is
weakly compatible with $S=X(R)$.

If $L$, $M$ are invertible sheaves and $s$ resp.\ $t$ are global
sections of $L$ resp.\ $M$, we denote by $st$ the product $s\otimes
t$, seen as a global section of $L\otimes M$.
\end{lab}

\begin{lab}
Let $L$, $L'$ be invertible sheaves on $X$, and let $\phi\colon
L\otimes L\isoto L'$ be an isomorphism. Fixing $\phi$, we can talk of
the sign of global (or local) sections of $L'$ at points $\xi\in
X(R)$.
Namely, given $f\in H^0(X,L')$ and $\xi\in X(R)$, the \emph{sign of
$f$ at $\xi$} (with respect to $\phi$) is defined as
$$\sgn_{\xi,\phi}(f)\>:=\>\sgn_\xi(a)\ \in\,\{-1,0,1\},$$
where $U\subset X$ is an open set with $\xi\in U(R)$ over which $L$
is trivial, $s\in H^0(U,L)$ is a generator of $L|_U$ and $a\in\scrO_X
(U)$ is defined by $f|_U=a\cdot\phi(s^2)$. We say that $f$ is
\emph{nonnegative}, resp.\ \emph{(strictly) positive} (with respect
to $\phi$), if $\sgn_{\xi,\phi}(f)\ge0$, resp.\ $>0$, for every $\xi
\in X(R)$. When $\phi$ is understood, we simply write $f(\xi)\ge0$ or
$f(\xi)>0$, instead of $\sgn_{\xi,\phi}(f)\ge0$ or $\sgn_{\xi,\phi}
(f)=1$, respectively.

It is clear that the definition of $\sgn_{\xi,\phi}(f)$ does not
depend on the choice of either $U$ or $s$. It does depend, however,
on the square root $L$ of $L'$ and on the choice of the isomorphism
$\phi$.
\end{lab}

\begin{rems}\label{convent}%
\hfil

1.\
Throughout this paper, we will only talk of signs of sections of
invertible sheaves $L'$ that are given in the form $L'=L\otimes L$.
These signs are always understood with respect ot the identity
isomorphism of $L\otimes L$. Therefore, we will consequently suppress
mentioning the isomorphism $\phi$. For example, if $s\in H^0(X,
L^{\otimes2})$ and $t\in H^0(X,M^{\otimes2})$, then for the signs of
$st\in H^0(X,(L\otimes M)^{\otimes2})$ we have $\sgn_\xi(st)=\sgn_\xi
(s)\cdot\sgn_\xi(t)$, for $\xi\in X(R)$.
\smallskip

2.\
Recall that $X$ denotes a (reduced separated) $R$-scheme of finite
type. If $L$ is an invertible sheaf on $X$ and $f\in H^0(X,L^
{\otimes2})$, then $S=\{\xi\in X(R)\colon f(\xi)\ge0\}$ is a closed
\sa\ subset of $X(R)$.
\end{rems}

\begin{dfn}\label{dividesects}%
Let $L$, $M$ be invertible sheaves on a scheme $X$, and let $s$
resp.\ $t$ be global sections of $L$ resp.\ $M$. We say that
\emph{$s$ divides $t$}, if there is a global section $u$ of $L^\du
\otimes M$ with $t=su$.
\end{dfn}

The following lemma is clear:

\begin{lem}\label{kuerzen}%
Let $L$, $M$ be invertible sheaves on $X$, let $s\in H^0(X,L)$ and
$t\in H^0(X,M)$.
\begin{itemize}
\item[(a)]
$s$ divides $t$ if and only if $X$ has a covering by open subsets
$U_\alpha$ such that $s|_{U_\alpha}$ divides $t|_{U_\alpha}$ for all
$\alpha$.
\item[(b)]
If $s$ is regular, and if $t'\in H^0(X,M)$ satisfies $st=st'$ in $H^0
(X,L\otimes M)$, then $t=t'$.
\qed
\end{itemize}
\end{lem}

\begin{lem}\label{primeigsects}%
Let $L$, $M$, $N$ be invertible sheaves on $X$ with regular global
sections $s$, $t$, $u$, respectively, and assume that $s$ divides
$tu$. If $Z(s)$ is reduced, and if $t$ does not vanish identically on
any irreducible component of $\supp(s)$, then $s$ divides $u$.
\end{lem}

\begin{proof}
We can localize und assume that $s,t,u$ are not zero divisors in the
ring $A$, and that $(s)=\sqrt{(s)}$. If $\p_1,\dots,\p_r$ are the
minimal prime ideals containing $s$, then the hypothesis says
$t\notin\p_1\cup\cdots\cup\p_r$. Since $tu\in(s)$ by hypothesis, we
conclude $u\in\bigcap_i\p_i=(s)$.
\end{proof}

The following lemma is the technical reason why weak compatibility is
a key notion for this paper:

\begin{lem}\label{keyproptotre}%
Let $X$ be an $R$-scheme of finite type, and let $S\subset X(R)$ be a
\sa\ set.
Let $L$, $M$ be invertible sheaves on $X$, and let $s\in H^0(X,L)$ be
a regular and reduced global section that is weakly compatible with
$S$.
\begin{itemize}
\item[(a)]
If $a$, $b\in H^0(X,M^{\otimes2})$ are nonnegative on $S$, and if $s$
divides $a+b$, then $s$ divides both $a$ and $b$.
\item[(b)]
If $a_1,\dots,a_r\in H^0(X,M)$ and $s$ divides $a_1^2+\cdots+a_r^2$,
then $s$ divides $a_1,\dots,a_r$.
\end{itemize}
\end{lem}

\begin{proof}
We may argue locally.
Let $x$ be a closed point of $X$ and let $A=\scrO_{X,x}$ be the local
ring of $X$ at $x$. By fixing trivializations of $L$ and $M$ at $x$
we may identify the values of local sections of these sheaves at the
stalk $x$ with elements of $A$. Let $f\in A$ correspond to $s$.
The assumption on $s$ means that $(f)=\p_1\cap\cdots\cap\p_r$ with
prime ideals $\p_i$ of $A$, and that for
$i=1,\dots,r$ there exists $\alpha_i\in\wt S\cap\Sper(A)$ with $\supp
(\alpha_i)=\p_i$.

For (a) let $g$, $h\in A$ correspond to $a$ resp.\ $b$. By
assumption, $a$ and $b$ are nonnegative on $\wt S$. For every $i=1,
\dots,r$ we have $g+h\in\p_i$, that is, $(g+h)(\alpha_i)=0$, and on
the other hand $g(\alpha_i)\ge0$ and $h(\alpha_i)\ge0$. Together this
implies $g$, $h\in\p_i$, and hence $g$, $h\in(f)$. For (b) let $g_j
\in A$ correspond to $a_j$. From (a) it follows that $s$ divides $a_j
^2$, whencee $g_j^2\in(f)$, for every $j$. Since $(f)=\sqrt{(f)}$ we
conclude $g_j\in(f)$.
\end{proof}


\section{Existence of totally real sections}\label{secexcompsec}%

The main result of this section is Theorem \ref{extotreschnitt},
which is an existence theorem for totally real global sections of
invertible sheaves on quasi-projective real varieties. Throughout
this section let $R$ denote a real closed base field.

The following lemma is an application of the (\sa) theorem on
implicit functions:

\begin{lem}\label{implicfctappl}%
Let $X$ be a reduced $R$-scheme of finite type and let $f\colon X\to
\A^m$ be an $R$-morphism. For $0\ne a\in R^m$ let $X_a\subset X$ be
the (scheme-theoretic) preimage under $f$ of the linear hyperplane
$\sum_{i=1}^ma_ix_i=0$ in $\A^m$. Let $0\ne a\in R^m$ be such that
$X_a$ contains no irreducible component of $X$, and assume that $X_a$
has a regular $R$-point $\xi$. Given an open neighborhood $U$ of
$\xi$ in $X(R)$, we have $U\cap(X_b)_\reg(R)\ne\emptyset$ for all
$b\in R^m$ close to $a$.
\end{lem}

\begin{proof}
Let $W\subset X\times\A^m$ be the closed subscheme whose $E$-valued
points are the pairs $(\eta,b)\in X(E)\times E^m$ with $\sum_ib_if_i
(\eta)=0$, for $E$ an $R$-algebra. The point $\xi$ is a regular point
of $X$,
and the tangent space to $W$ at $(\xi,a)$ consists of the pairs
$(u,v)\in T_\xi(X)\oplus R^m$ for which $\bil{d_\xi f_a}u+\bil
{f(\xi)}v=0$. Here $d_\xi f_a\in T_\xi(X)^*$ denotes the differential
of $f_a:=\sum_ia_if_i$ at $\xi$.
Since $\xi$ is a regular point of $X_a$ (and $X_a$ doesn't contain a
Zariski neighborhood of $\xi$) we have $d_\xi f_a\ne0$, and so $(\xi,
a)$ is a regular point of $W$. The projection $\pi\colon W\to\A^m$ is
submersive at $(\xi,a)$.
By the (semi-algebraic) theorem on implicit functions (\cite{BCR}
Cor.\ 2.9.8) there is a local continuous (semi-algebraic) section
$b\mapsto(\eta(b),b)$ of $\pi$ around $a$ with $\eta(a)=\xi$. So
$\eta(b)$ is an $R$-point of $X_b$ and is, by continuity, a regular
point of $X_b$, for $b$ close to $a$.
\end{proof}

\begin{lem}\label{transvsectoff}%
Let $X$ be a reduced $R$-scheme of finite type, let $L$ be an
invertible sheaf on $X$, and let $V\subset H^0(X,L)$ be a
finite-dimensional linear subspace. Moreover let open \sa\ subsets
$U_1,\dots,U_r$ of $X(R)$ be given. Then the set of all regular
sections $s\in V$ with $Z(s)_\reg\cap U_i\ne\emptyset$ for $i=1,
\dots,r$ is open in $V$ (in the euclidean topology).
\end{lem}

\begin{proof}
Being a regular global section is an open condition on $s$. Given
$s\in V$ with $Z(s)_\reg\cap U_i\ne\emptyset$ for $i=1,\dots,r$, it
follows from Lemma \ref{implicfctappl} that any $t\in V$ sufficiently
close to $s$ satisfies $Z(t)_\reg\cap U_i\ne\emptyset$ for $i=1,
\dots,r$ as well.
\end{proof}

\begin{lem}\label{offmengetransvreg}%
Let $k$ be a field and let $X$ be a reduced quasi-projective
$k$-scheme without $0$-dimensional irreducible components. Let $L$,
$M$ be invertible sheaves on $X$, with $M$ ample. Let finitely many
closed points $\xi_1,\dots,\xi_r$ and $\eta_1,\dots,\eta_p$ in $X$ be
given, with $\xi_i\ne\eta_j$ for all $i,j$, and assume that $\xi_1,
\dots,\xi_r$ are regular points of $X$. There is an integer $n_0\ge0$
such that, for any $n\ge n_0$, the sheaf $L\otimes M^{\otimes n}$ has
a regular global section $s$ for which $\xi_1,\dots,\xi_r$ are
regular points of $Z(s)$ and $\eta_1,\dots,\eta_p\notin Z(s)$.
\end{lem}

\begin{proof}
Assume the lemma is shown in the case where $M$ is very ample. In the
general case there is $d\ge1$ such that $M^{\otimes d}$ is very
ample. For $i=1,\dots,d$ there exists, by the assumption, an integer
$n_i\in\N$ such that, for $n\ge n_i$, $L\otimes M^{\otimes(i+dn)}$
has a regular section as required. Therefore, if $n\ge\max\{i+dn_i
\colon i=1,\dots,d\}$, then $L\otimes M^{\otimes n}$ has a regular
section as required.

So we may assume that $X\subset\P^m$ is locally closed and $M=\scrO_X
(1)$. Let $\ol X\subset\P^m$ be the reduced closure of $X$. There
exists a coherent sheaf $L'$ on $\ol X$ for which $L'|_X\cong L$
(\cite{H} exercise II.5.15).
By adding more points to the sequence $\eta_1,\dots,\eta_p$ if
necessary we may assume that every irreducible component of $X$
contains one of the $\eta_j$.
Consider the $0$-dimensional closed subscheme
$$Y\>:=\>\coprod_{i=1}^r\Spec(\scrO_{X,\xi_i}/\m_{X,\xi_i}^2)\>\amalg
\>\coprod_{j=1}^p\Spec k(\eta_j)$$
of $X$, where $k(\eta_j)=\scrO_{X,\eta_j}/\m_{X,\eta_j}$ denotes the
residue field of $\eta_j$,
and let $i\colon Y\into\ol X$ be the inclusion. We have the exact
sequence
$$0\to F\to L'\to i_*i^*L'\to0$$
of coherent sheaves on $\ol X$.
There exists $n_0\ge0$ such that $H^1(\ol X,F(n))=0$ for all $n\ge
n_0$ (\cite{H} III.5.2). Hence the restriction map
$$H^0(\ol X,L'(n))\to H^0(Y,L'(n))$$
is surjective for $n\ge n_0$. In particular, for every $n\ge n_0$,
there is a section $s\in H^0(\ol X,L'(n))$ that vanishes in each of
$\xi_1,\dots,\xi_r$, but only of first order, and that does not
vanish in any of $\eta_1,\dots,\eta_p$. It is clear that $s|_X$ is a
regular global section of $L(n)$,
and that $\xi_1,\dots,\xi_r$ are nonsingular points of $Z(s|_X)$.
\end{proof}

We recall the following consequences of Bertini's theorems.

\begin{prop}\label{berti}%
Let $k$ be a field of characteristic zero, and let $X$ be a reduced
and locally closed $k$-subscheme of $\P^n$.
\begin{itemize}
\item[(a)]
For almost all $k$-hyperplanes $H$ the intersection $H\cap X$ is
reduced.
\item[(b)]
If $X$ is geometrically irreducible and $\dim(X)\ge2$, then for
almost all $k$-hyperplanes $H$ the intersection $H\cap X$ is
irreducible.
\end{itemize}
\end{prop}

Here, as usual, ``for almost all $k$-hyperplanes $H$'' means: For
all $k$-hyperplanes that lie in a non-empty open subset of the
dual projective space $(\P^n_k)^*$.

\begin{proof}
See \cite{J} Cor.\ 6.11.
\end{proof}

\begin{thm}\label{extotreschnitt}%
Let $R$ be a real closed field and $X$ a reduced quasi-projective
$R$-scheme with $\dim(X')\ge2$ for every irreducible component $X'$
of $X$. Let $S\subset X(R)$ be a \sa\ set that is Zariski dense
in~$X$.
\begin{itemize}
\item[(a)]
If $L$, $M$ are invertible sheaves on $X$ and $M$ is ample, there
exists $n_0\ge0$ such that, for every $n\ge n_0$, $L\otimes M^
{\otimes n}$ has a regular and reduced global section $s$ that is
weakly compatible with $S$.
\item[(b)]
If $X$ is irreducible, any very ample invertible sheaf on $X$ has a
nonzero reduced global section that is weakly compatible with $S$.
\end{itemize}
\end{thm}

\begin{proof}
(a)
As in the proof of Lemma \ref{offmengetransvreg}, we can assume that
$M$ is very ample. So we may assume that $X\subset\P^n$ is locally
closed and $M=\scrO_X(1)$. Let $X_1,\dots,X_r$ be the irreducible
components of $X$. Since $S$ is Zariski dense in $X$, there exists.
for each $i=1,\dots,r$, a non-empty open set $U_i\subset X_\reg(R)
\cap X_i(R)$ contained in $S$ (Proposition \ref{artinlang}). For
$i=1,\dots,r$ fix a point $\xi_i\in U_i$. Choose $n_0\ge0$ so large
that the conclusion of Lemma \ref{offmengetransvreg} holds for
$\xi_1,\dots,\xi_r$, and that $L(n)$ is very ample for all $n\ge
n_0$. Fix $n\ge n_0$ and a regular section $s\in H^0(X,L(n))$ for
which $\xi_1,\dots,\xi_r$ are regular points of $Z(s)$. Choose a
finite-dimensional linear subspace $V\subset H^0(X,L(n))$ containing
$s$ such that the pair $(L(n),V)$
defines a locally closed embedding $X\into\P^{\dim(V)-1}$. Inside $V$
there is an open neighborhood $\Omega$ of $s$ (with respect to the
euclidean topology) consisting of regular sections $t$ satisfying
$Z(t)_\reg\cap U_i\ne\emptyset$ for $i=1,\dots,r$, by Lemma
\ref{transvsectoff}. On the other hand, for generic $t\in V$,
Bertini's theorem \ref{berti} tells us that the scheme $Z(t)$ is
reduced and $Z(t)\cap X_i$ is irreducible for $i=1,\dots,r$.
Indeed, $\dim(X_i)\ge2$ by assumption, and $X_i$ is geometrically
irreducible since $X_i$ has a regular $R$-point.
In particular, there exists such $t$ in $\Omega$. For any such $t$,
the reduced closed subscheme $Z(t)$ of $X$ has precisely $r$
different irreducible components $Z_1,\dots,Z_r$, for which $Z_i
\subset X_i$ and $(Z_i)_\reg\cap U_i\ne\emptyset$ ($i=1,\dots,r$).
Hence $Z_i(R)\cap S$ is Zariski dense in $Z_i$ for each $i$ (see
\ref{artinlang}), and so $t$ is weakly compatible with $S$ (see
\ref{dfndivwc}).

(b)
We can assume $X\subset\P^n$ and $L=\scrO_X(1)$. There is a non-empty
open subset $U$ of $X_\reg(R)$ contained in $S$. Any linear
hyperplane $H$ that intersects $X$ transversely in some point of $U$
and for which $H\cap X$ is reduced and irreducible corresponds to a
section of $L$ that is weakly compatible with $S$.
\end{proof}

\begin{rem}
If $X$ is reducible in Theorem \ref{extotreschnitt}, statement (b)
usually fails. As an illustration let $X=X_0\cup\cdots\cup X_n$ be a
union of isotropic quadrics $X_i$ in $\P^n_R$ with signature $(n,1)$
such that the sets $X_i(R)$ are pairwise disjoint and not nested. If
the diameters of these quadrics are sufficiently small, there is no
hyperplane meeting each $X_i$ in a real point.
\end{rem}

\begin{rem}\label{remextotreschnitt}%
In both parts of Theorem \ref{extotreschnitt}, we may require in
addition that $s$ does not vanish in a finite list of given closed
points of $X$. (See Lemma \ref{transvsectoff}.)
\end{rem}

\begin{cor}\label{verallgrogg}%
Let $X$ be a reduced quasi-projective $R$-scheme with $\dim(X')\ne1$
for every irreducible component $X'$ of $X$, and let $S\subset X(R)$
be a \sa\ subset that is Zariski dense in $X$.
\begin{itemize}
\item[(a)]
Any divisor on $X$ is linearly equivalent to a difference $D_1-D_2$,
where $D_1$, $D_2$ are reduced effective divisors that are weakly
compatible with $S$.
\item[(b)]
If $X$ is affine, any divisor is linearly equivalent to an effective
reduced divisor that is weakly compatible with $S$.
\end{itemize}
\end{cor}

\begin{proof}
Let $D$ be a divisor with associated invertible sheaf $\scrO_X(D)$,
and choose a very ample invertible sheaf $M$ on $X$. If $n>0$ is
large enough then $M^{\otimes n}$ and $\scrO_X(D)\otimes M^{\otimes
n}$ have regular and reduced global sections $s$ resp.\ $t$ that are
weakly compatible with $S$ (Theorem \ref{extotreschnitt}). Hence $D$
is linearly equivalent to $Z(t)-Z(s)$. If $X$ is affine we can take
$M=\scrO_X$ and $s=1$.
\end{proof}

\begin{rem}
Particular cases of Corollary \ref{verallgrogg} were proved by
Roggero \cite{Ro}. For $R=\R$ and $S=X(\R)$, she proved the case
where $X$ is normal of dimension $\ge2$, and either affine or
projective.
\end{rem}

\begin{rem}\label{curvrems1}%
The exclusion of one-dimensional irreducible components in these
results cannot be avoided. More precisely, the following is known.
Let $X$ be a connected nonsingular projective curve over $R$ with
$X(R)\ne\emptyset$. There exists an integer $n\ge0$ such that any
divisor $D$ on $X$ with $\deg(D)>n$ is linearly equivalent to an
effective totally real divisor (\cite{sch:tams} Cor.\ 2.10; see also
Monnier \cite{Mo1} Thm.\ 3.6 who proves $n\le2g-1$ for $M$-curves or
$(M-1)$-curves over $\R$). More generally, we have, for $R=\R$:
\end{rem}

\begin{prop}\label{curvecase}%
Let $X$ be a connected nonsingular projective curve over $\R$, and
let $S\subset X(\R)$ be an infinite \sa\ subset with $S\cap O\ne
\emptyset$ for every oval $O$ of $X(\R)$. There is an integer $n\ge0$
such that every invertible sheaf $L$ on $X$ with $\deg(L)\ge n$ has a
nonzero global section that is weakly compatible with~$S$.
\end{prop}

\begin{proof}
Let $O_1,\dots,O_r$ be the different ovals of $X(\R)$. For $i=1,
\dots,r$, fix a point $Q_i\in S\cap O_i$. Let $J$ be the Jacobian of
$X$, and let $J(\R)_0$ be the identity component of the real Lie
group $J(\R)$. The class of a Weil divisor $D=\sum_Pn_PP$ of degree
zero on $X$ lies in $J(\R)_0$ if and only if, for any oval $O$ of
$X(\R)$, the sum $\sum_{P\in O}n_P$ is even (\cite{sch:tams} Lemma
2.6). By \cite{sch:tams} Lemma 2.12 there is an integer $m\ge1$ such
that, for every $\alpha\in J(\R)_0$, there exist $m$ points $P_1,
\dots,P_m\in S$ with $\alpha=\sum_{i=1}^m[P_i-Q_1]$. We claim that
the assertion of the proposition is satisfied with $n:=m+r-1$.
Indeed, let $D$ be a Weil divisor with $\deg(D)=d\ge m+r-1$. There
exist (unique) numbers $a_2,\dots,a_r\in\{0,1\}$ such that the class
of $D-dQ_1-\sum_{i=2}^ra_i(Q_i-Q_1)$ lies in $J(\R)_0$. Hence there
exist $P_1,\dots,P_m\in S$ such that this divisor is equivalent to
$\sum_{j=1}^m[P_j-Q_1]$. Altogether we conclude
$$D\>\sim\>\Bigl(d-m-\sum_{i=2}^ra_i\Bigr)Q_1+\sum_{i=2}^ra_iQ_i+
\sum_{j=1}^mP_j.$$
\end{proof}

\begin{rems}\label{curvrems2}%
\hfil

1.\
If $S$ fails to meet one of the ovals, the assertion of Proposition
\ref{curvecase} clearly becomes false.
If $R$ is a non-archimedean real closed field, \ref{curvecase} may
also fail when $S$ intersects each oval in an open set, c.f.\
\cite{sch:tams} Rem.\ 2.15.
\smallskip

2.\
Part (b) of Theorem \ref{extotreschnitt} clearly fails in general,
even for $R=\R$ and $S=X(\R)$.
\smallskip

3.\
Now let $X/R$ be a singular integral projective curve with $|X(R)|=
\infty$. If the singularities of $X$ are real nodes (possibly with
nonreal tangents), Monnier \cite{Mo2} has proved (for $R=\R$) the
existence of an integer $n\ge0$ such that every divisor of degree
$\ge n$ is linearly equivalent to some totally real divisor supported
by regular points. On the other hand, there are singular curves with
ample divisors no multiple of which is equivalent to a totally real
divisor, as the next proposition shows.
\end{rems}

\begin{prop}
Let $X$ be an integral projective curve over $R=\R$ with $X(\R)\ne
\emptyset$. Assume that $X$ is singular, but all real points on $X$
are regular. Then for any $d\ge1$ there exists an invertible sheaf
$L$ on $X$ of degree $d$ such that, for any $n\ge1$, $L^{\otimes n}$
does not have any nonzero totally real global section.
\end{prop}

\begin{proof}
Let $J$ be the generalized Jacobian of $X$. From the assumption it
follows that the real Lie group $J(\R)$ is not compact. Fix a point
$P_0\in X(\R)$, and consider the map $\varphi\colon X(\R)\to J(\R)$,
$\varphi(P)=[P-P_0]$, the class of the locally principal Weil divisor
$P-P_0$. ($\varphi$~extends to a morphism of varieties, but we only
need that $\varphi$ is a continuous map.) The image of $\varphi$ is a
compact subset of $J(\R)$. Fix $d\ge1$, and consider the set $TR_d$
of all divisors $D$ of degree $d$ for which there exists $n\ge1$ such
that $nD$ is equivalent to an effective totally real divisor. We show
that the subset
$$\{[D-dP_0]\colon D\in TR_d\}$$
of $J(\R)$ is compact, thereby proving the proposition. Let $D\in
TR_d$. There are $n\ge1$ and $P_1,\dots,P_{nd}\in X(\R)$ with $n[D]=
\sum_{i=1}^{nd}[P_i]$, and hence $n[D-dP_0]=\sum_{i=1}^{nd}\varphi
(P_i)$. Let $M_d\subset J(\R)$ denote the set of $d$-fold sums of
elements of $\im(\varphi)$. Then $M_d$ is a compact set, and $n[D-d
P_0]$ is a sum of $n$ elements of $M_d$. Therefore the claim follows
from the next lemma.
\end{proof}

\begin{lem}
Let $(G,+)$ be a commutative real Lie group with finitely many
connected components. Given a subset $M$ of $G$, let
$$C\>=\>\bigl\{x\in G\colon\ex n\ge1\ \ex y_1,\dots,y_n\in M\ nx=y_1
+\cdots+y_n\bigr\}.$$
If $\ol M$ is compact, then $\ol C$ is compact as well.
\end{lem}

\begin{proof}
$G$ is isomorphic to a direct product $K\times V$ where $K$ is a
compact Lie group and $V$ is a real vector space of finite dimension.
It suffices to prove that $\ol C+K$ is compact, so we may assume $G=
V=\R^n$. But then $C$ is contained in the convex hull of $M$, and so
the assertion is clear.
\end{proof}

\begin{lab}
As for affine curves, we only make the following remark. Let $X$ be
a nonsingular affine curve over $\R$, and let $S\subset X(\R)$ be a
non-empty open \sa\ subset that intersects every connected component
of $X(\R)$. Then every invertible sheaf on $X$ has a nonzero global
section that is weakly compatible with $S$. Indeed, let $X\subset
\ol X$ be the nonsingular completion of $X$.
Let $L$ be an invertible sheaf on $X$, and choose an invertible sheaf
$L_1$ on $\ol X$ with $L_1|_X\cong L$.
By Riemann-Roch, there exists an ample invertible sheaf $M$ on
$\ol X$ such that $M|_X$ is trivial.
By Proposition \ref{curvecase}, $L_1\otimes M^{\otimes n}$ has a
section on $\ol X$ that is weakly compatible with $S$, for some $n\ge
0$. Hence $L$ has such a section on $X$.
\end{lab}


\section{Projective Positivstellensatz}\label{secppss}%

In the following it is essential that the real closed base field is
$\R$, the usual real numbers (or a real closed subfield thereof). The
main result is:

\begin{thm}\label{mainthmtent}%
Let $X$ be a reduced projective $\R$-scheme without one-dimensional
irreducible components. Let $L$, $M$, $N_1,\dots,N_r$ be invertible
sheaves on $X$, with $M$ ample, and let regular global sections $h_i
\in H^0(X,N_i^{\otimes2})$ ($i=1,\dots,r$) be given such that the
\sa\ subset
$$K\>=\>\bigl\{\xi\in X(\R)\colon h_1(\xi)\ge0,\dots,h_r(\xi)\ge0
\bigr\}$$
of $X(\R)$ is Zariski dense in $X$. Let moreover $f\in H^0(X,L^
{\otimes2})$ and $g\in H^0(X,M^{\otimes2})$ be given with $f(\xi)>0$
and $g(\xi)>0$ for every $\xi\in K$. Then there is $n_0\ge0$ such
that, for all $n\ge n_0$, there exist sums of squares $s_e$ with
\begin{equation}\label{eqpo}%
fg^n\>=\>\sum_{e\in\{0,1\}^r}s_e\cdot h_1^{e_1}\cdots h_r^{e_r}
\end{equation}
in $H^0(X,L^{\otimes2}\otimes M^{\otimes2n})$. (More precisely,
$s_e$ is a sum of squares of elements of $H^0(X,F_e)$ with $F_e=
L\otimes M^{\otimes n}\otimes\bigotimes_iN_i^{-e_i}$, for $e\in
\{0,1\}^r$.)
\end{thm}

See convention \ref{convent}.1 for the meaning of signs of sections
at points $\xi\in X(\R)$.

\begin{proof}
Under the assumptions of the theorem, we will prove the following
apparently weaker statement: Given $f'\in H^0(X,L^{\otimes2})$ and
$g'\in H^0(X,M')$, where $M'$ is ample and $f'(\xi)>0$, $g'(\xi)\ne0$
for all $\xi\in K$, there exists $n\ge0$ such that $f'g'^{2n}$ allows
an identity \eqref{eqpo}. Put $M'=M^{\otimes2}$ and apply this weaker
statement to $(f',g')=(f,g)$ and $(f',g')=(fg,g)$, to get Theorem
\ref{mainthmtent}.

Resetting notation, assume for the rest of this proof that $g\in H^0
(X,M)$ vanishes nowhere on $K$. We have to show that there is an
identity \eqref{eqpo} with some even number $n\ge0$. We can
immediately dispense with $0$-dimensional irreducible components of
$X$, and can therefore assume $\dim(X')\ge2$ for every irreducible
component $X'$ of $X$. Since we can replace $M$ by $M^{\otimes m}$
(and $g$ by $g^m$) for $m\ge1$, we may assume that $M$ is very ample.
So let $X\subset\P^n$ be a closed subscheme, and let $M=\scrO_X(1)$.
We consider the open affine subscheme $Y=X_g=\{x\in X\colon g(x)\ne
0\}$ of $X$. The ring $\R[Y]=H^0(Y,\scrO_X)$ of regular functions on
$Y$ is the homogeneous localization of the graded ring $S=\bigoplus_
{k\ge0}H^0(X,\scrO_X(k))$ by the element $g\in S_1$. In other words,
the elements of $\R[Y]$ are the fractions $\frac a{g^k}$ with $k\ge0$
and $a\in\Gamma(X,\scrO_X(k))$, with the usual rules for such
fractions.
By the hypothesis on $g$ we have $K\subset Y(\R)$. Note that $K$ is
closed in $X(\R)$ (\ref{convent}.2), therefore $K$ is a compact
\sa\ subset of $Y(\R)$.

For any invertible sheaf $P$ on $X$ that is generated by its global
sections, there exists $s\in H^0(X,P^{\otimes2})$ with $s(\xi)>0$ for
all $\xi\in X(\R)$. Indeed, if $P$ is generated by its global
sections $s_1,\dots,s_m$, then $s:=s_1^2+\cdots+s_m^2$ has this
property.

By Theorem \ref{extotreschnitt}(a) there exists an integer $t\ge0$
such that the invertible sheaf $L^\du(t)$ on $X$ has a regular and
reduced global section $h$ that is weakly compatible with $K$. By
Remark \ref{remextotreschnitt}, we can get in addition that $h$ does
not vanish identically on any irreducible component of $\supp(h_i)$,
for any $i\in\{1,\dots,r\}$.
Since the $h_i$ are regular, this means that also conversely no $h_i$
vanishes identically on any irreducible component of $\supp(h)$.
Consider
$$\varphi\>:=\>\frac{fh^2}{g^{2t}}\ \in\R[Y],$$
a regular function on $Y$ that is nonnegative on $K$.

For each $i=1,\dots,r$ we choose an integer $a_i\ge0$ such that the
sheaf $N_i':=N_i^\du(a_i)$ is generated by global sections, and we
choose $p_i\in H^0(X,N_i'^{\otimes2})$ such that $p_i>0$ on $X(\R)$.
Then
$$H_i\>:=\>\frac{h_ip_i}{g^{2a_i}}\ \in\R[Y],$$
and
$$K\>=\>\bigl\{\xi\in Y(\R)\colon H_1(\xi)\ge0,\dots,H_r(\xi)\ge0
\bigr\}.$$
In particular, the compact subset $K$ of $Y(\R)$ is basic closed.

We shall prove that $\varphi$ lies in the preordering generated by
$H_1,\dots,H_r$ in the ring $\R[Y]$. Before doing so we show how to
complete the proof of the theorem. By assumption, there are integers
$d_e\ge0$ and sums of squares $s_e\in H^0(X,\scrO_X(2d_e))$, for
$e\in\{0,1\}^r$, such that
$$\frac{fh^2}{g^{2t}}\>=\>\sum_{e\in\{0,1\}^r}\frac{s_e}{g^{2d_e}}
\cdot\frac{(h_1p_1)^{e_1}\cdots(h_rp_r)^{e_r}}{g^{2(a_1e_1+\cdots+
a_re_r)}}$$
in $\R[Y]$. After multiplying with a sufficiently high even power of
$g$ we get an identity
\begin{equation}\label{equngek}%
fh^2g^{2n}\>=\>\sum_{e\in\{0,1\}^r}t_eh_1^{e_1}\cdots h_r^{e_r}
\end{equation}
in $H^0(X,\scrO_X(2(t+n)))$, where $n\ge0$ and $t_e$ is $s_e$ times
an even power of $g$ and a product of some of the $p_i$.

Since the reduced section $h$ of $L^\du(t)$ is weakly compatible with
$K$, and since each summand $t_eh_1^{e_1}\cdots h_r^{e_r}$ on the
right is nonnegative on $K$, it follows from Lemma
\ref{keyproptotre}(a) that $h$ divides $t_eh_1^{e_1}\cdots h_r^{e_r}$
for all $e$. Since no $h_i$ vanishes identically on any irreducible
component of $\supp(h)$, it follows from Lemma \ref{primeigsects}
that $h$ divides $t_e$, for all $e$. Applying Lemma
\ref{keyproptotre}(b) we conclude that there are sums of squares
$t'_e$ with $t_e=h^2t'_e$, for $e\in\{0,1\}^r$. We can therefore
cancel $h^2$ from identity \eqref{equngek}, see Lemma \ref{kuerzen},
and get
$$fg^{2n}\>=\>\sum_et'_eh_1^{e_1}\cdots h_r^{e_r},$$
as desired.

It remains therefore to show that $\varphi$ lies in the preordering
$T$ of $\R[Y]$ that is generated by $H_1,\dots,H_r$. For this we use
the local-global criterion of \cite{sch:surf} Cor.\ 2.10. Note that
this criterion does apply here since the preordering $T$ of $\R[Y]$
is archimedean, $K$ being compact (see, e.g., \cite{PD} Thm.\
5.1.17). For any closed point $x$ of $Y$ we have to show $\varphi\in
T_x$, where $T_x$ denotes the preordering generated by $H_1,\dots,
H_r$ in the local ring $\scrO_{X,x}$.

Let $u$ be a local generator of $L^\du(t)$ at $x$. In $\scrO_{X,x}$
we have
$$\varphi\>=\>\frac f{g^{2t}}\cdot h^2\>=\>\frac{fu^2}{g^{2t}}\cdot
\Bigl(\frac hu\Bigr)^2,$$
and both factors $\varphi_1=\frac{fu^2}{g^{2r}}$ and $\frac hu$ are
in $\scrO_{X,x}$. It suffices to show $\varphi_1\in T_x$. By the
hypotheses we have $\varphi_1>0$ on $\wt K\cap\Sper(\scrO_{X,x})$,
which is the basic closed subset of $\Sper(\scrO_{X,x})$ associated
with the preordering $T_x$. From \cite{sch:lpo} I.\,Prop.\ 2.1 it
follows that $\varphi_1\in T_x$, and the proof is complete.
\end{proof}

It is worthwile to isolate the unconstrained case (no $h_i$):

\begin{cor}\label{mainunconstr}%
Let $X$ be a reduced projective $\R$-scheme without irreducible
components of dimension one for which $X(\R)$ is Zariski dense in
$X$. Let $L$, $M$ be invertible sheaves on $X$, and let $f$ resp.\
$g$ be strictly positive global sections of $L^{\otimes2}$ resp.\
$M^{\otimes2}$. If $M$ is ample then $fg^n$ is a sum of squares of
global sections of $L\otimes M^{\otimes n}$ for all sufficiently
large $n$.
\qed
\end{cor}

\begin{rem}
Suppose that, in the situation of Theorem \ref{mainthmtent}, we are
in the particular case where $L=M^{\otimes k}$ for some $k\ge0$. Then
we get identity \eqref{eqpo} for all sufficiently large $n\equiv k$
(mod~$2$) without the Zariski density hypothesis on $K$, and without
any condition on the irreducible components of $X$. Indeed, we can
directly apply Schm\"udgen's Positivstellensatz \cite{Sm} to the
regular function $\varphi=\frac f{g^k}$ on $X_g$, which is strictly
positive on $K$, and do not need to multiply with a factor~$h$.
\end{rem}

\begin{rem}\label{curvetrue}%
Theorem \ref{mainthmtent} remains true when $X$ is a nonsingular
projective curve and $|K\cap O|=\infty$ for every oval $O$ of
$X(\R)$. Indeed, Theorem \ref{extotreschnitt}(a) (plus Remark
\ref{remextotreschnitt}) is true in this case by Proposition
\ref{curvecase}. (We have actually not verified that $s$ in
\ref{extotreschnitt}(a) can also be chosen to be reduced, and leave
it to the reader to show that this is true.)
We do not know, however, if Theorem \ref{mainthmtent} holds true
without the dimension restriction on the irreducible components of
$X$, or without the condition that $K$ is Zariski dense in $X$.
\end{rem}

\begin{cor}\label{urform2}%
Write $\x=(x_0,\dots,x_n)$. Let $h_1,\dots,h_r\in\R[\x]$ be
homogeneous polynomials of even degree, and let
$$S\>=\>\{\xi\in\R^{n+1}\colon h_1(\xi)\ge0,\dots,h_r(\xi)\ge0\}.$$
Assume there is $\xi\in\R^{n+1}$ with $h_i(\xi)>0$ for $i=1,\dots,r$.
If $f$, $g\in\R[\x]$ are homogeneous of (even) positive degree and
strictly positive on $S\setminus\{(0,\dots,0)\}$, then $fg^m$ lies in
the preordering generated by $h_1,\dots,h_r$, for all sufficiently
large $m\ge0$.
\end{cor}

\begin{proof}
That is, $fg^m$ satisfies an identity \eqref{eqpo}. This is just an
affine reformulation of Theorem \ref{mainthmtent} in the case $X=
\P^n$. Indeed, since $S$ has non-empty interior in $\R^{n+1}$, the
corresponding projectivized \sa\ subset of $\P^n(\R)$ is Zariski
dense in $\P^n$.
\end{proof}

\begin{rem}\label{ururform}%
In particular, if $f$, $g\in\R[\x]$ are positive definite forms of
positive degree, then $fg^m$ is a sum of squares of forms for all
$m\gg0$. In this form the result was recently applied to time
continuous dynamical systems by Ahmadi and Parrilo \cite{AP}. Using
the result, the authors prove that, whenever a homogeneous polynomial
Lyapunov function exists, there also exists one (of possibly higher
degree) which is a sum of squares and whose negative derivative is
also a sum of squares.
\end{rem}

In \cite{St}, Stengle proved that there exist positive semidefinite
forms $f$ in $\R[\x]$ of which no odd power $f^{2m+1}$ is a sum of
squares of forms. It follows directly from our result that this
cannot happen when $f$ is strictly positive definite:

\begin{cor}\label{oddpow}%
Let $f\in\R[\x]$ be a strictly positive definite form. Then there
exists an odd number $m\ge1$ such that $f^m$ is a sum of squares of
forms.
\qed
\end{cor}

\begin{rem}
There is no such result for inhomogeneous strictly positive
polynomials. For example, the polynomial $f=x^3+(xy^2-x^2-1)^2$ in
$\R[x,y]$ is easily seen to be strictly positive on $\R^2$, but no
odd power of $f$ is a sum of squares in $\R[x,y]$, according to
Stengle \cite{St}.
\end{rem}

\begin{lab}
As another concrete example we may apply the Positivstellensatz to
products of projective spaces, or closed subvarieties thereof. For
example, if $\x=(x_0,\dots,x_m)$ and $\y=(y_0,\dots,y_n)$ are
coordinate tuples and $f$, $g\in\R[\x,\y]$ are positive definite
bihomogeneous forms, where the bidegree $(d_1,d_2)$ of $g$ satisfies
$d_1>0$ and $d_2>0$, then $fg^N$ is a sum of squares of bihomogeneous
forms for all $N\gg0$. Similarly for multihomogeneous forms.
\end{lab}

From Theorem \ref{mainthmtent}, it is possible to derive a necessary
and sufficient condition, in terms of sums of squares, for a global
section to be strictly positive. To keep the formulation simpler we
restrict to the unconstrained case, however the statement holds with
constraints as well. It also holds when $X$ is a nonsingular
projective curve (Remark \ref{curvetrue}).

\begin{cor}\label{ppss}%
Let $X$ be a reduced projective $\R$-scheme whose irreducible
components have dimension $\ne1$, let $L$ be an invertible sheaf on
$X$, and let $f\in H^0(X,L^{\otimes2})$. Then $f$ is strictly
positive if and only if there is a very ample sheaf $M$ for which
$L\otimes M$ is generated by global sections and for which the
following condition holds for some basis $g_1,\dots,g_r$ of $H^0
(X,M)$:

$(*)$ For any $h\in H^0(X,L\otimes M)$ there exist a real number
$\epsilon>0$ and an integer $n\ge0$ such that $g^n(gf-\epsilon h^2)$
is a sum of squares of sections of $L\otimes M^{n+1}$, where $g:=g_1
^2+\cdots+g_r^2$.

If $f>0$, condition $(*)$ is satisfied for any ample $M$ and any
sequence $g_1,\dots,g_r$ in $H^0(X,M)$ that generates $M$.
\end{cor}

\begin{proof}
First assume $f>0$, let $M$ be an ample invertible sheaf generated by
global sections $g_1,\dots,g_r$, and let $h\in H^0(X,L\otimes M)$.
Writing $g:=g_1^2+\cdots+g_r^2$, there exists $\epsilon>0$ such that
$gf-\epsilon h^2>0$. Indeed, this is clear locally on $X(\R)$, and by
compactness of $X(\R)$ it is true globally. By Theorem
\ref{mainthmtent} (or Corollary \ref{mainunconstr}, in this case),
there exists $n\ge0$ such that $g^n(gf-\epsilon h^2)$ is a sum of
squares of sections of $L\otimes M^{n+1}$.

Conversely, assume that $M$ is ample and generated by global sections
$g_1,\dots,g_r$, that $L\otimes M$ is generated by global sections as
well, and that $(*)$ holds for any $h\in H^0(X,L\otimes M)$. Given
any point $\xi\in X(\R)$, we choose $h$ such that $h(\xi)\ne0$, and
conclude $f(\xi)>0$.
\end{proof}

In the case of nonsingular curves or surfaces, we even have a
projective Nichtnegativstellensatz, due to the fact that $\rm psd=
sos$ holds in regular local rings of dimension $\le2$:

\begin{thm}\label{projnnss}%
Let $X$ be a connected nonsingular projective curve or surface over
$\R$ with $X(\R)\ne\emptyset$. Let $L$, $M$ be invertible sheaves on
$X$, with $M$ ample, and let $f\in H^0(X,L^{\otimes2})$ and $g\in H^0
(X,M^{\otimes2})$ be global sections with $f\ge0$ and $g>0$
everywhere on $X(\R)$. Then $fg^N$ is a sum of squares for all
sufficiently large $N\ge0$.
\end{thm}

\begin{proof}
One proceeds as in the proof of Theorem \ref{mainthmtent}. For
$\dim(X)=2$ the proof can be copied word by word up to the last line.
Due to the weaker assumption, we now only know $\varphi_1\ge0$
(rather than $\varphi_1>0$) on $\Sper(\scrO_{X,x})$. Replacing the
reference to \cite{sch:lpo} 2.1 by \cite{sch:local} Thm.\ 4.8
resolves the matter. When $\dim(X)=1$, use Remark \ref{curvrems1}
instead of Theorem \ref{extotreschnitt} to get the existence of $h$.
The last step in the proof of \ref{mainthmtent} becomes elementary
(an element in a discrete valuation ring $B$ containing $\frac12$
that is a sum of squares in the field of fractions of $B$ is a sum of
squares in $B$).
\end{proof}

\begin{rem}
Various generalizations of Theorem \ref{projnnss} are possible that
we won't explicate here in detail. For example, in the curve case we
can allow ordinary real nodes on $X$ (use Monnier's theorem mentioned
in \ref{curvrems2}, plus \cite{sch:local} Thm.\ 3.9). In the surface
case we can allow nonreal singular points on $X$, provided their
local rings are factorial (use \cite{sch:local} Cor.\ 4.9).
We can also give constrained versions of the Nichtnegativstellensatz
\ref{projnnss}, provided the constraints $h_i$ satisfy local
regularity conditions as in \cite{sch:surf} Thm.\ 3.2, or more
generally, as provided by the results of \cite{sch:lpo} I and~II. For
all these generalizations, it only matters that the preordering $T_x$
in $\scrO_{X,x}$ is saturated, for every closed point $x$ of $Y=X_g$
(see the end of the proof of Theorem \ref{mainthmtent}).
\end{rem}


\end{document}